\documentclass[11pt,oneside]{amsart}
\usepackage{amsmath, ifthen, amsfonts, amssymb, srcltx, amsopn, xcolor, enumerate, comment} 
\usepackage[linktocpage=true]{hyperref}
\usepackage[all]{xy}
\usepackage{pb-diagram, pb-xy}
\usepackage{overpic}
\usepackage{tikzsymbols}
\usepackage{tikz-cd}

\usepackage[T1]{fontenc}

\usepackage[parfill]{parskip}    
\usepackage{fullpage}
\usepackage{graphicx}
\usepackage{amssymb,amsmath,latexsym}
\usepackage{tikz}
\usetikzlibrary{arrows,decorations.markings,decorations.pathreplacing,calc,matrix,intersections}
\tikzset{->-/.style={decoration={markings,mark=at position #1 with {\arrow{>}}},postaction={decorate}}}

\usepackage{epstopdf}
\usepackage{ifpdf}
\usepackage{color}
\usepackage{float}
\usepackage{amscd,amsmath, mathabx}
\usepackage{amssymb,amsmath,latexsym,color,enumerate,tikz}
\usepackage[all]{xypic}       
\usepackage[varg]{pxfonts}
\usepackage{amsmath,amsthm,amsfonts,amssymb,fancyhdr,graphics,relsize,tikz-cd,mathtools,faktor,tikz,changepage}
\usepackage{graphicx}
\usepackage{placeins}

\usepackage{dsfont}

\newcommand{\showcommentsbox}{yes}

\newsavebox{\commentbox}
{\ifthenelse{\equal{\showcommentsbox}{yes}}%
{\footnotemark
        \begin{lrbox}{\commentbox}
        \begin{minipage}[t]{1.25in}\raggedright\sffamily\tiny
        \footnotemark[\arabic{footnote}]}
{\begin{lrbox}{\commentbox}}}%
{\ifthenelse{\equal{\showcommentsbox}{yes}}%
{\end{minipage}\end{lrbox}\marginpar{\usebox{\commentbox}}}
{\end{lrbox}}}

\definecolor{red}{rgb}{1,0,0} 

 \definecolor{darkgreen}{rgb}{0, .7, 0}

 \definecolor{purple}{rgb}{.7, 0, 1}

\tikzset{mynode/.style={draw,circle,fill=black,inner sep=2pt,outer sep=0.5pt}}

\newtheorem*{thm:main}{Theorem \ref{thm:main}}
\newtheorem*{corollaryMain}{Corollary \ref{cor:mainRAAGs}}
\newtheorem*{corollary1}{Corollary \ref{corollary1}}
\newtheorem*{corollary2}{Corollary \ref{corollary2}}
\newtheorem*{corollary3}{Corollary \ref{corollary3}}
\newtheorem*{corollary4}{Corollary \ref{corollary4}}

\newtheorem{theorem}{Theorem}[section]
\newtheorem*{theorem*}{Theorem}
\newtheorem*{lemma*}{Lemma}
\newtheorem{proposition}[theorem]{Proposition}
\newtheorem{lemma}[theorem]{Lemma}
\newtheorem{corollary}[theorem]{Corollary}

\theoremstyle{definition}
\newtheorem{definition}[theorem]{Definition}

\newtheorem{question}[theorem]{Question}

\theoremstyle{remark}
\newtheorem{remark}[theorem]{Remark}

\usepackage{hyperref}
 
\begin{document}

\title[fg normal subgroups of RAAGs]{On finitely generated normal subgroups of right-angled Artin groups and graph products of groups}

\author[M. Casals-Ruiz]{Montserrat Casals-Ruiz}
\address{Ikerbasque - Basque Foundation for Science and Matematika Saila,  UPV/EHU,  Sarriena s/n, 48940, Leioa - Bizkaia, Spain}
\email{montsecasals@gmail.com}

\author[J. Lopez de Gamiz Zearra]{Jone Lopez de Gamiz Zearra}
\address{Mathematics Institute, University of Warwick, Coventry, United Kingdom and Matematika Saila,  UPV/EHU,  Sarriena s/n, 48940, Leioa - Bizkaia, Spain}
\email{Jone.Lopez-De-Gamiz@warwick.ac.uk}

\begin{abstract}
A classical result of Schreier (see \cite{Schreier}) states that nontrivial finitely generated normal subgroups of free groups are of finite index, that is, free groups can only quotient to finite groups with finitely generated kernel. 

In this note we extend this result to the class of right-angled Artin groups (RAAGs). More precisely, we prove that the quotient of a RAAG by a finitely generated (full) normal subgroup is abelian-by-finite and finite-by-abelian.

As Schreier's result extends to nontrivial free products of groups, we further show that our result extends to graph products of groups.

As a corollary, we deduce, among others, that finitely generated normal subgroups of RAAGs have decidable word, conjugacy and membership problems and that they are hereditarily conjugacy separable.
\end{abstract}

\date{\today}
\maketitle

\section{Introduction}

A classical result of Schreier (see \cite{Schreier}) states that nontrivial finitely generated normal subgroups of free groups are of finite index, that is, free groups can only quotient to finite groups with finitely generated kernel.

The main goal of this note is to extend Schreier's result for right-angled Artin groups (RAAGs for short). More precisely, we describe the quotient of a RAAG by a finitely generated normal subgroup and deduce finiteness properties and the decidability of algorithmic problems for finitely generated normal subgroups of RAAGs.

The class of RAAGs extends the class of finitely generated free groups, by allowing relations stating that some of the generators commute. Subgroups of RAAGs play an important role in geometric group theory by providing examples of groups with interesting finiteness properties and algorithmic behaviour.

Any RAAG $GX$ admits a decomposition $GX_1 \times \dots \times GX_k$ as a direct product where $GX_i$ is a directly indecomposable RAAG. A subgroup $H< G_1 \times \cdots \times G_n$ is called \emph{full} if $H$ intersects nontrivially each factor, i.e. $H \cap G_i \neq 1$ for all $i\in \{1, \dots, n\}$. Notice that if $H$ is full, in particular it is nontrivial and if $H$ intersects trivially one of the factors, say $H\cap G_1=1$, then $H$ is (isomorphic to) a subgroup of $G_2 \times \dots \times G_n$. 

We can next formulate the analogue of Scheirer's theorem for free groups in the context of RAAGs.

\begin{corollaryMain}
Let $G$ be a RAAG and let $N<G$ be a finitely generated full normal subgroup of $G$. Then $G/N$ is virtually abelian.
\end{corollaryMain}

If one does not require the subgroup $N$ to be full, we obtain the following result:

\begin{corollaryMain}
Let $G$ be a RAAG and let $N<G$ be a finitely generated normal subgroup of $G$. Then $G/N$ is virtually a RAAG.
\end{corollaryMain}

Notice that the theorem is optimal in the sense that there are RAAGs that fiber, i.e. there are RAAGs that quotient to free abelian groups with finitely generated kernel (see \cite{MeierMeinertVanWyk}).

This strong restriction on the class of quotients of RAAGs with finitely generated kernel does not pass to finitely presented subgroups of RAAGs. Indeed, Haglund and Wise showed that the Rips' Construction can be adjusted to virtually compact special groups and so the class of quotients of virtually compact special groups with finitely generated kernel coincides with the class of finitely presented groups.

\begin{theorem}[\cite{HaglundWise}, see also \cite{Arenas}]
For any finitely presented group $Q$,
there exists a virtually compact special group $H$ that quotients to $Q$ with finitely generated kernel. In other words, there exists a subgroup $S$ of a RAAG $G$ which quotients to a finite index subgroup of $Q$ with finitely generated kernel.
\end{theorem}

Corollary \ref{cor:mainRAAGs} states that the quotient of a RAAG by a finitely generated full normal subgroup is abelian-by-finite. In fact, we deduce that the quotient is finite-by-abelian by showing that a finite index abelian subgroup of the quotient is actually central. As a consequence, we get that finitely generated full normal subgroups of RAAGs are finite index subgroups of kernels of homomorphisms to abelian groups. In the literature, these subgroups are sometimes called of \emph{Stallings-Bieri type}.

\begin{corollaryMain}
Let $G$ be a RAAG and let $N<G$ be a finitely generated full normal subgroup of $G$. Then $N$ is commensurable with a kernel of a character.

More precisely, there is $[\chi] \in \Sigma^1(G)$ such that $N< \ker \chi$ and $N$ has finite index in $\ker \chi$.
\end{corollaryMain}

In \cite{BridsonMiller}, Bridson and Miller proved an analogous result for normal subdirect products of direct products of groups. Recall that a subgroup $H< A_1 \times \cdots \times A_n$ is a  \emph{subdirect product} if each projection homomorphism from $A_1 \times \cdots \times A_n$ to $A_i$ restricts to an epimorphism from $H$ to $A_i$. They showed that if $G < A_1 \times \cdots \times A_n$ is a normal subdirect product, then $G$ is the kernel of a homomorphism from $A_1 \times \cdots \times A_n$ to an abelian group (see \cite[Proposition 1.2]{BridsonMiller}). Note, however, that being a subdirect product is quite restrictive: if $N_i$ is a proper finitely generated normal subgroup of $A_i$, then $N_1\times \cdots \times N_n$ is a finitely generated normal subgroup of $A_1\times \cdots A_n$ but it is not a subdirect product. It is clear that the condition of being a subdirect product is essential in that generality while we do not need to require it for RAAGs.

Direct products of finitely many free groups are themselves RAAGs and their subgroups have been studied extensively due to their rich finiteness properties and their relevance in the study of residually free groups. Baumslag and Roseblade studied subgroups of the direct product of two free groups in \cite{BaumslagRoseblade} and they showed that there is a dichotomy between finitely generated and finitely presented subgroups. More concretely, they proved that finitely presented subgroups of the direct product of two free groups are virtually the direct product of two free groups. In contrast, they showed that finitely generated subgroups of the direct product of two free groups have a complex structure and there are continuously many of them up to isomorphism.

One of the corollaries that we obtain from Corollary \ref{cor:mainRAAGs} is that when restricted to finitely generated normal subgroups, the set is in fact countable:

\begin{corollary1}
The set of finitely generated normal subgroups of right-angled Artin groups is countable.
\end{corollary1}

As we mentioned before, subgroups of RAAGs are also interesting from the algorithmic point of view. Mihailova proved in \cite{Mihailova} that there are finitely generated subgroups of the direct product of two free groups with undecidable membership problem and Miller showed in \cite{Miller} that they also have undecidable conjugacy problem. From the work of Baumslag and Roseblade, however, it follows that these algorithmic problems are decidable for finitely presented subgroups of the direct product of two free groups. Bridson, Howie, Miller and Short conducted a programme that culminated in the papers \cite{BridsonHowieMillerShort} and \cite{BridsonHowieMillerShort2} to generalise the result of Baumslag and Roseblade to finitely presented subgroups of direct products of finitely many free groups (and, more generally, limit groups). Among the several powerful results that they obtained, a remarkable one is that the word, conjugacy and membership problems are decidable for finitely presented subgroups of direct products of finitely many free/limit groups (although the decidability of the isomorphism problem is still open). One may wonder if this nice algorithmic behaviour can be extended from the class of finitely presented subgroups of direct products of free groups to the class of finitely presented subgroups of RAAGs. In \cite{Bridson} Bridson gave a negative answer to this question by showing that there is a RAAG $A$ and a finitely presented subgroup $S$ of $A\times A$ with undecidable conjugacy and membership problems. This shows that finite presentability is not enough to ensure a good algorithmic behaviour and raises the question of whether there is an appropriate property for a subgroup of a RAAG to have decidable algorithmic problems.

Combining Corollary \ref{cor:mainRAAGs} with results from \cite{Bridson2}, we prove the following:

\begin{corollary2} 
The word, conjugacy and membership problems are decidable for finitely generated normal subgroups of RAAGs.
\end{corollary2}

The isomorphism problem for the class of finitely generated normal subgroups of RAAGs is still open and we record it in the following:

\begin{question}
Is the isomorphism problem decidable for finitely generated normal subgroups of RAAGs?
\end{question}

For finitely presented groups, the decidability of the word and conjugacy problems are a consequence of their residual properties, namely, residual finiteness and conjugacy separability correspondingly (see \cite{Malcev}). Recall that a group $G$ is \emph{conjugacy separable} if for any two non-conjugate elements $x,y \in G$ there is a homomorphism $\varphi$ from $G$ to a finite group $Q$ such that $\varphi(x)$ and $\varphi(y)$ are not conjugate in $Q$. While residual finiteness passes to subgroups, conjugacy separability does not, even to finite index subgroups. This motivates the following definition. A group is said to be \emph{hereditarily conjugacy separable} if each of its finite index subgroups is conjugacy separable. Minasyan proves in \cite{Minasyan} that RAAGs are hereditarily conjugacy separable. However, together with Martino, they showed that this is not true for finitely presented subgroups of RAAGs, that is, there is a finitely presented subgroup of a RAAG which is conjugacy separable but not hereditarily conjugacy separable (see \cite[Theorem 1.1]{MartinoMinasyan}).

As before, finite presentability is not enough to inherit hereditary conjugacy separability from the RAAG but results from \cite{Minasyan} together with Corollary \ref{cor:mainRAAGs} lead to the fact that normal subgroups of RAAGs do inherit this property:

\begin{corollary3}
Any finitely generated full normal subgroup of a RAAG is hereditarily conjugacy separable, so in particular, conjugacy separable.
\end{corollary3}

Let $G$ be a group and let $Aut(G)$ be the group of automorphisms of $G$. We say that $\psi \in Aut(G)$ is \emph{pointwise inner} if for each $g\in G$, $\psi(g)$ is conjugate to $g$. The set of pointwise inner automorphisms, $Aut_{pi}(G)$, forms a normal subgroup of $Aut(G)$, and clearly the subgroup of inner automorphisms, $Inn(G)$, is contained in $Aut_{pi}(G)$. Grossman showed in \cite{Grossman} that if $G$ is a finitely generated conjugacy separable group with $Aut_{pi}(G)= Inn(G)$, then $Out(G)$ is residually finite.

Antolin, Minasyan and Sisto showed in \cite[Theorem 1.6]{AntolinMinasyanSisto} that if $H$ is a non-abelian finitely generated subgroup of a RAAG, then $Aut_{pi}(H)=Inn(H)$.

From Corollary \ref{corollary3} we have that finitely generated full normal subgroup of RAAGs are conjugacy separable. 

Combining these results together with Grossman's theorem, we deduce the following:

\begin{corollary}
The outer automorphism $Out(N)$ of a full finitely generated normal subgroup of a RAAG is residually finite.
\end{corollary}

Bestvina and Brady, by means of Morse theory, studied the finiteness properties of kernels of homomorphisms from RAAGs to $\mathbb{Z}$ that send all the standard generators to $1$ and characterised the finiteness properties of those kernels in terms of the topology of the flag complex associated to the RAAG (see \cite{BestvinaBrady}). After that, Meier, Meinert and VanWyk computed in \cite{MeierMeinertVanWyk} the Bieri-Neumann-Strebel-Renz invariants for RAAGs and as a consequence, they determined the finiteness properties of co-abelian normal subgroups of RAAGs. We translate those results to finitely generated normal subgroups of RAAGs:

\begin{corollary4}
Let $G$ be a RAAG given by its standard presentation and let $N< G$ be a finitely generated full normal subgroup of $G$. Then there is an algorithm that given a finite generating set of $N$ as words in the generators of $G$ and $n\in \mathbb N$, decides whether or not $N$ is of type $F_n$ and $FP_n$.
\end{corollary4}

In \cite{DicksLeary}, Dicks and Leary describe explicit presentations for the kernels of the above homomorphisms. We wonder whether that result can be extended to all finitely presented normal subgroups of RAAGs in an algorithmic way:

\begin{question}
Is there an algorithm that, given a finite generating set of a finitely presented full normal subgroup of a RAAG, computes its presentation?
\end{question}

Schreier's result was generalised by Baumslag to nontrivial free products of groups (see \cite{Baumslag}), namely, if $G=G_1 \ast \cdots \ast G_k$ with $k \geq 2$ and $G_i \neq 1$ for $i\in \{1,\dots,k\}$ and if $N<G$ is a nontrivial finitely generated normal subgroup of $G$, then $N$ has finite index in $G$. 

In a similar way, we extend our main result, Corollary \ref{cor:mainRAAGs}, to graph products of groups as follows.

Let $\Gamma$ be a finite simplicial graph and let $\mathcal{G}= \mathcal G(\Gamma, \{G_v\}_{v\in V(\Gamma)})$ be the graph product with underlying graph $\Gamma$ and nontrivial vertex groups $G_v$. We say that $v\in V(\Gamma)$ is \emph{central} if $v$ is adjacent to every vertex of $\Gamma$. If $v$ is central, we say that $G_v$ is a \emph{central vertex group}. Every finite graph $\Gamma$ can be decomposed as a finite join of subgraphs, say $\Gamma_1, \dots, \Gamma_k$, such that $\Gamma_i$ is indecomposable as a join of two subgraphs (see Definition \ref{definition1}). This decomposition leads to a decomposition of $G$ as a direct product $G_1\times \cdots \times G_k$, where $G_i$ is the graph product $\mathcal{G}(\Gamma_i, \{G_v\}_{v\in V(\Gamma_i)})$.

\begin{thm:main}
Let $G$ be a graph product of groups and let $N<G$ be a finitely generated full normal subgroup of $G$. If each central vertex group satisfies that all its proper quotients are abelian-by-finite and finite-by-abelian, then $G \slash N$ is abelian-by-finite and finite-by-abelian.
\end{thm:main}



\begin{corollary}
Let $G$ be a graph product of groups without central vertex groups and let $N<G$ be a finitely generated full normal subgroup of $G$. Then $G \slash N$ is abelian-by-finite and finite-by-abelian.
\end{corollary}

Notice that the result for RAAGs (Corollary \ref{cor:mainRAAGs}) follows from Theorem \ref{thm:main} since RAAGs are graph products with infinite cyclic vertex groups and so all central vertex groups have only abelian quotients. Right-angled Coxeter groups (RACGs) are also graph products with cyclic groups of order two as vertex groups. In particular, proper quotients of vertex groups are finite and so we obtain:

\begin{corollary}\label{corollary5}
Let $G$ be a RACG and let $N<G$ be a finitely generated full normal subgroup of $G$. Then $G \slash N$ is abelian-by-finite and finite-by-abelian.
\end{corollary}

The classical result of Schreier has been established in more general settings. We have mentioned that Baumslag extended it to nontrivial free products of groups, and Karrass and Solitar extended it to finitely generated subgroups of free groups that contain a nontrivial normal subgroup:

\begin{theorem}\cite[Theorem 1]{KarrassSolitar}
Let $F$ be a free group and let $H$ be a finitely generated subgroup containing a nontrivial normal subgroup of $F$. Then $H$ must be of finite index in $F$.
\end{theorem}

The main result of this note states that finitely generated full normal subgroups of RAAGs are co-(virtually abelian). Thus, having the previous theorem of Karrass and Solitar in mind, one could wonder whether a similar statement holds for RAAGs, namely, one could ask if $G$ is a RAAG and $H$ is a finitely generated subgroup of $G$ containing a nontrivial normal subgroup $N$ of $G$, then $H$ virtually contains the commutator subgroup $[G,G]$. However, Bridson, Howie, Miller and Short gave examples of finitely presented subgroups of direct products of finitely many free groups containing some term of the lower central series of the group but not the commutator subgroup (see \cite[Theorem H]{BridsonHowieMillerShort}). This shows that the most naive generalisation does not hold. However, we would like to propose the following:

\begin{question}
Let $G$ be a RAAG, let $H$ be a finitely generated subgroup of $G$ and $N$ a full normal (not necessarily finitely generated) subgroup of $G$ such that $N<H$. Is it true that there is a subnormal series
\[ H_0 < H_1 < \cdots < H_{n-1} < H_n =G\]
such that $H_{i+1}\slash H_i$ is either finite or cyclic and $H_0$ is of finite index in $H$?
\end{question}

There are several conditions that assure that a group belongs to the class of groups for which nontrivial finitely generated normal subgroups are either finite or of finite index. For instance, this is the case for groups of deficiency greater than one (see \cite[Remark B3.25]{Strebel}). Gaboriau generalised this result in \cite[Theorem 6.8]{Gaboriau} by showing that it suffices to have positive first $L_2$-Betti number. Recently, Kielak has proved that a finitely generated RFRS group $G$ is virtually fibred, in the sense that it admits a virtual surjection to $\mathbb{Z}$ with a finitely generated kernel, if and only if the first $L_2$-Betti number of $G$ vanishes (see \cite{Kielak}).

We take the converse point of view in this characterisations and ask:

\begin{question}

Can one characterise the class of groups such that the quotient of the group by any nontrivial finitely generated normal subgroup is virtually abelian?
\end{question}




\section{Graph products}

We begin this section recalling some basics on graphs.

\begin{definition}\label{definition1}
In this note all graphs $\Gamma=(V,E)$ are assumed to be \emph{simplicial}, that is, $(v,v)\notin E$ and for all $v,w\in V$, there is at most one edge $(v,w)\in E$. A graph $\Delta=(V^\prime,E^\prime)$ is an \emph{induced subgraph} of $\Gamma$ if $V^\prime < V$ and for all $v,w\in V^\prime$, $(v,w)\in E$ if and only if $(v,w)\in E^\prime$.

Given two graphs $\Gamma_1=(V_1, E_1)$ and $\Gamma_2=(V_2,E_2)$, the \emph{join} of $\Gamma_1$ and $\Gamma_2$ is the graph with vertex set $V=V_1 \cup V_2$ and edges $E=E_1 \cup E_2 \cup \{(v,w) \mid v\in V_1, w\in V_2\}$. In other words, $\Gamma_1$ and $\Gamma_2$ are induced subgraphs of $\Gamma$ and each vertex in $\Gamma_1$ is connected to each vertex in $\Gamma_2$.
\end{definition}

We say that $v\in \Gamma$ is \emph{central} if $\Gamma$ is the join of the induced subgraphs $\{v\}$ and $\Gamma\setminus \{v\}$. We denote by $Z(\Gamma)$ the set of central vertices of $\Gamma$.  

The \emph{link} $lk(v)$ of a vertex $v$ in $\Gamma$ is the induced subgraph of $\Gamma$ with vertex set $\{ w\in V \mid (v,w)\in E\}$. The \emph{star} $st(v)$ of a vertex $v$ in $\Gamma$ is the induced subgraph of $\Gamma$ defined as the join of $\{v\}$ and $lk(v)$.

Let $\Gamma=(V,E)$ be a finite simplicial graph. Let $\mathcal G=\mathcal G(\Gamma, \{G_v\}_{v\in V})$ be a graph product of groups. For a subgraph $\Delta<\Gamma$, we denote by $\mathcal G(\Delta)$ the graph product induced by the graph $\Delta$ as a subgroup of $\mathcal G$. 

For each $v\in V$, the group $\mathcal G$ splits as follows:
\begin{equation}\label{eq:splitting}
\mathcal G= \mathcal G( \Gamma \setminus \{v\}) \ast_{\mathcal G(lk(v))} \mathcal G(st(v)),
\end{equation}
where $\Gamma \setminus \{v\}$ is the full subgraph of $\Gamma$ with vertex set $V \setminus \{v\}$ and $lk(v)$ and $st(v)$ are the link and the star of $v$ in $\Gamma$, respectively. Notice that this splitting is proper if and only if $v \notin Z(\Gamma)$. In this case, the associated Bass-Serre tree $T$ is reduced and $G$ acts minimally on $T$. We call $T$ the \emph{standard tree of} $\mathcal G$ (\emph{associated to $v$}).

\begin{lemma}\label{lem:WPD graph products}\cite[Corollary 6.19]{MinasyanOsin}
Suppose that $\Gamma$ has more than one vertex and that the graph product $\mathcal G$ is directly indecomposable. Let $T$ be a standard tree for $\mathcal G$. Then $\mathcal G$ contains a WPD element for the given action of $\mathcal G$ on $T$.
\end{lemma}

\begin{remark}\label{rem:generalised WPD}
Suppose that $\Gamma$ is the join of the graphs $\Gamma_1, \dots, \Gamma_k$ and $\Gamma_i$ is not a join. Then $\mathcal G$ decomposes as the direct product $\mathcal G(\Gamma_1) \times \cdots \times \mathcal G(\Gamma_k)$ and $\mathcal G(\Gamma_i)$ is directly indecomposable if $\Gamma_i$ is not a central vertex. Let $v\notin Z(\Gamma)$ and suppose that $v\in \Gamma_i$. Let $T$ be the standard tree for $\mathcal{G}(\Gamma_i)$ associated to $v$. Then $\mathcal{G}$ acts on $T$ with the property that $\mathcal G(\Gamma_1) \times \cdots \times  \mathcal G(\Gamma_{i-1}) \times \mathcal G(\Gamma_{i+1}) \times \cdots \times \mathcal G(\Gamma_k)$ is in the kernel of the action and so the epimorphism $\mathcal G\to \mathcal G(\Gamma_i)$ induces an equivariant map on $T$. 

Let $h$ be the WPD element associated to the action of $\mathcal G(\Gamma_i)$ on $T$ given by Lemma \ref{lem:WPD graph products}. Then, for each vertex group $A$ associated to the action of $\mathcal G$ on $T$, we have that $A \cap A^h$ is the kernel of the action, that is, \[A \cap A^h = \mathcal G(\Gamma_1) \times \cdots \times  \mathcal G(\Gamma_{i-1}) \times \mathcal G(\Gamma_{i+1}) \times \cdots \times \mathcal G(\Gamma_k).\]
\end{remark}

It easily follows from Lemma \ref{lem:WPD graph products} that if $N < \mathcal G$ is a finitely generated full normal subgroup, then it acts minimally on $T$ as we record in the following result:

\begin{proposition}\label{prop:min tree N}
Let $\mathcal G$ be a graph product of groups, let $T$ be a standard tree for $\mathcal G$ associated to a vertex $v\notin Z(\Gamma)$ and let $N$ be a finitely generated full normal subgroup of $\mathcal G$. Then $N$ acts minimally on $T$ and the graph $N \backslash T$ is finite.
\end{proposition}

\begin{proof}
Let $\mathcal G=\mathcal G(\Gamma_1) \times \cdots \times \mathcal G(\Gamma_k)$ be the direct product decomposition induced from the join decomposition of $\Gamma$. In particular, $\mathcal G(\Gamma_i)$ is directly indecomposable for $\Gamma_i$ not a central vertex. 

Since by assumption $v\notin Z(\Gamma)$, then $v\in \Gamma_i$ for some $i$ and $\Gamma_i$ has more than one vertex. Let $T$ be the standard tree associated to $v$.

The group $N$ is a subgroup of $\mathcal G$, so it also acts on the tree $T$. Suppose that all the elements of $N$ are elliptic. Since the group $N$ is finitely generated, this implies that $N$ is a subgroup of a vertex group, say $A$. 

Let $h\in \mathcal G$ be a WPD element for the action of $\mathcal G(\Gamma_i)$ on $T$, which exists by Lemma \ref{lem:WPD graph products}. Then, $N= N^h$ is a subgroup of $A^h$, so 
\[
N < A \cap A^h.
\]
But the intersection $A \cap A^h$ is contained in the kernel of the action, so it lies inside the direct product $\mathcal G(\Gamma_1) \times \dots \times  \mathcal G(\Gamma_{i-1}) \times \mathcal G(\Gamma_{i+1}) \times \dots \times \mathcal G(\Gamma_k)$ (see Remark \ref{rem:generalised WPD}) and this contradicts the fact that $N$ is full.

Hence, $N$ contains a hyperbolic element and there is $T_0$ a unique $N$-invariant subtree. Since $N$ is normal in $\mathcal G$, the $N$-invariant subtree $T_0$ is also invariant under the action of $G$. But $T$ is minimal as a $\mathcal G$-tree, so $T_0= T$.

Since the group $N$ is finitely generated by assumption and $N$ acts minimally on $T$, by \cite[Propositon 7.9]{Bass}, the graph $N \backslash T$ is finite.
\end{proof}

We next prove the main result.

\begin{theorem}\label{thm:main}
Let $\mathcal G=\mathcal G(\Gamma, \{G_v\}_{v\in V\Gamma})$ be a graph product of groups and let $N< \mathcal{G}$ be a finitely generated full normal subgroup of $\mathcal{G}$. If each central vertex group $G_v$ satisfies that all proper quotients are abelian-by-finite and finite-by-abelian, then $G\slash N$ is both finite-by-abelian and abelian-by-finite, i.e. we have the following exact sequences:
$$
1 \to A \to G\slash N \to Q \to 1
$$
and
$$
1\to Q^{\prime} \to G\slash N \to A^{\prime} \to 1
$$
where $A$ and $A^{\prime}$ are abelian and $Q$ and $Q^{\prime}$ are finite groups.
\end{theorem}

\begin{proof}

The graph product $\mathcal G$ is of the form 
\[
\prod_{v\in Z(\Gamma)} G_v \times \mathcal G(\Gamma\setminus Z(\Gamma)).
\]
Since by assumption $N$ is full, we have that $N \cap G_v\ne 1$ for all $v\in Z(\Gamma)$. Then $N \cap G_v$ is a normal subgroup (possibly not finitely generated) of $G_v$ for $v\in Z(\Gamma)$. By the assumption on the central vertex groups, we have that $G_v\slash (N \cap G_v)$ is abelian-by-finite and finite-by-abelian.

Let $p$ be the epimorphism $p:\mathcal{G} \mapsto \mathcal G(\Gamma\setminus Z(\Gamma))$ and let $B=\prod_{v\in Z(\Gamma)} G_v$ be the kernel of $p$. We next show that the quotient group $\mathcal G \slash N$ is isomorphic to the direct product $B \slash (B\cap N) \times \mathcal G(\Gamma\setminus Z(\Gamma)) \slash p(N)$. Indeed, consider the following short exact sequence:
\[
1 \to BN \slash N \to \mathcal{G} \slash N \to \frac{\mathcal G \slash N }{BN \slash N} \to 1.
\]
Note that $BN= B \times p(N)$, so
\begin{equation}\label{eq:isomorphism}
\frac{\mathcal G \slash N }{BN \slash N} \cong \mathcal G \slash BN = \mathcal G \slash (B\times p(N)) \cong \mathcal G(\Gamma\setminus Z(\Gamma)) \slash p(N),
\end{equation}
and the isomorphism sends an element $yp(N) \in \mathcal G(\Gamma\setminus Z(\Gamma)) \slash p(N)$ to $(1,y)(N) \frac{BN}{N} \in \frac{\mathcal G \slash N }{BN \slash N}$.

Moreover, from the isomorphism \eqref{eq:isomorphism} we also get that the exact sequence right splits. In addition, $BN \slash N \cong B \slash (B \cap N)$, so $BN \slash N$ is generated by elements of the form $(x,1)N$ for $x\in B$. As a consequence, the group $\mathcal G \slash N$ is isomorphic to the direct product $BN \slash N \times \mathcal G(\Gamma\setminus Z(\Gamma)) \slash p(N)$.

Since $BN \slash N\cong B \slash (B \cap N)$, $BN \slash N$ is a quotient of $B$ and by the assumption on the central vertex groups, it follows that it is abelian-by-finite and finite-by-abelian. Therefore, it suffices to show that $\mathcal G(\Gamma\setminus Z(\Gamma)) \slash p(N)$ is abelian-by-finite and finite-by-abelian. The group $p(N)$ is a nontrivial finitely generated full normal subgroup of $\mathcal G(\Gamma\setminus Z(\Gamma))$ and $\Gamma \setminus Z(\Gamma)$ is a graph without central vertices. Hence, it suffices to prove the statement for graph products whose defining graph does not have central vertices and, without loss of generality, we further assume that $\Gamma$ does not have central vertices.

Let $v\in V(\Gamma)$. Consider the splitting described in Equation \eqref{eq:splitting} and let $T$ be the standard tree of $\mathcal{G}$ associated to $v$. Now, by Proposition \ref{prop:min tree N}, the graph $N \backslash T$ is finite, so in particular, the number of edges in the graph needs to be finite. Therefore,
\[ \big{|} N \backslash \mathcal G \slash \mathcal G(lk(v)) \big{|} < \infty.\]
The subgroup $N$ is normal in $\mathcal G$, so the double coset $N \backslash \mathcal G \slash \mathcal G(lk(v))$ coincides with the single coset $\mathcal G \slash N \mathcal G(lk(v)).$ Hence, since $\big{|}\mathcal G \slash (N\mathcal G(lk(v))) \big{|} < \infty$, the subgroup $N\mathcal G(lk(v))$ has finite index in $\mathcal G$.

Since the number of vertices in $\Gamma$ is finite, we have that
\[
H=\bigcap\limits_{v\in V(\Gamma)} N \mathcal G(lk(v))
\]
has also finite index in $\mathcal G$ for being a finite intersection of finite index subgroups.

We next show that $H \slash N$ is central in $G \slash N$. Let $g_v$ be a generator of the vertex group $G_v$ and let $h\in H$. We want to show that $[h,g_v] \in N$. Since $h\in H=\bigcap\limits_{v\in V(\Gamma)} N \mathcal G(lk(v))$, there is $n\in N$ and $h_v\in \mathcal G(lk(v))$ such that $h=nh_v$. Then
\[
[h,g_v]=[nh_v, g_v]=[n,g_v]^{h_v} [h_v,g_v]
\]
where the last equality follows from the commutator identities. Since $h_v\in \mathcal G(lk(v))$ and $g_v\in G_v$, we have that $[g_v,h_v]=1$. Furthermore, $[n,g_v]^{h_v} \in N$ since $n\in N$ and $N$ is a normal subgroup. It follows that $[h,g_v] \in N$ for each $v\in V(\Gamma)$, each generator $g_v$ of the vertex group $G_v$ and each $h\in H$. Hence $H\slash N$ is central in $G\slash N$. Since $H$ is of finite in $G$, $H\slash N$ is of finite index index in $G\slash N$. It follows that $G\slash N$ is virtually abelian. Furthermore, the subgroup $H\slash N$ is central and of finite index in $G\slash N$, so the center of $G\slash N$ has also finite index in $G \slash N$. As a consequence, by the result of Schur in \cite{Schur} we have that the commutator subgroup of $G\slash N$, namely $[G,G]N\slash N$, is finite and so $G\slash N$ is also finite-by-abelian.
\end{proof}

Since RAAGs and RACGs are special cases of graph products of groups satisfying the hypothesis of Theorem \ref{thm:main}, we obtain the following:

\begin{corollary}\label{cor:mainRAAGs}
Let $G$ be a RAAG (respectively a RACG).

Let $N<G$ be a finitely generated full normal subgroup of $G$. Then $G\slash N$ is finite-by-abelian and abelian-by-finite.

Let $N<G$ be a finitely generated normal subgroup of $G$. Then $G \slash N$ is virtually a RAAG (respectively a RACG).
\end{corollary}

\section{Corollaries of the main result}

In this section we focus on the class of RAAGs and prove some interesting consequences for finitely generated normal subgroups of RAAGs. Notice that most of the corollaries also hold for the class of RACGs.

\begin{corollary}\label{corollary1}
The set of finitely generated normal subgroups of right-angled Artin groups is countable.
\end{corollary}

\begin{proof}
Any finitely generated normal subgroup is full in some RAAG and since the set of RAAGs is countable, it suffices to show that the set of finitely generated full normal subgroups of a given RAAG $G$ is countable.

From Corollary \ref{cor:mainRAAGs} we get that finitely generated full normal subgroups of $G$ are of finite index in kernels onto abelian groups. 

Let us show first that the set of finitely generated co-abelian normal subgroups of $G$ is countable. Indeed, there is a one-to-one correspondence between normal co-abelian subgroups and subgroups of the abelianisation of $G$. Since the abelianisation of $G$ is a free abelian group of finite rank, the number of subgroups is countable, and so is the set of finitely generated normal co-abelian subgroups of $G$. 

Now, the set of finite index subgroups of a finitely generated group is countable, and this concludes the proof.
\end{proof}

The decidability of the algorithmic problems is deduced also from the specific structure of quotients of RAAGs with finitely generated kernel:

\begin{corollary} \label{corollary2}
The word, conjugacy and membership problems are decidable for finitely generated normal subgroups of RAAGs.
\end{corollary}

\begin{proof}
The decidability of the word problem is inherited by subgroups and RAAGs do have decidable word problem.

Note that the membership problem for $N$ in a group $G$ is decidable if and only if the word problem for $G \slash N$ is decidable. By Corollary \ref{cor:mainRAAGs}, if $G$ is a RAAG and $N< G$ is a finitely generated normal subgroup, then $G \slash N$ is virtually a RAAG, so the membership problem is also decidable.

Finally, for the conjugacy problem, Bridson proved in \cite{Bridson2} that a normal subgroup of a bicombable group $G$ has decidable conjugacy problem, provided $G \slash N$ has decidable generalised word problem. RAAGs are bicombable since they act properly and cocompactly on CAT(0) spaces. Moreover, if $G$ is a RAAG and $N$ is a finitely generated full normal subgroup of $G$, then $G \slash N$ is virtually abelian (see Corollary \ref{cor:mainRAAGs}), so in particular, $G \slash N$ has decidable generalised word problem and this concludes the proof.
\end{proof}

Minasyan independently showed in \cite{Minasyan} that if $G$ is a RAAG, $N$ is a finitely generated normal subgroup of $G$ and $G \slash N$ is subgroup separable, then $N$ has decidable conjugacy problem (see \cite[Theorem 2.8]{Minasyan}). Thus, Corollary \ref{corollary2} might be also viewed as a consequence of that result. In fact, Minasyan proved that result by studying the conjugacy separability of normal subgroups of RAAGs:

\begin{theorem}\cite[Corollary 11.3]{Minasyan}
Let $N$ be a finitely generated normal subgroup of a right-angled Artin group $G$ such that the quotient $G \slash N$ is virtually polycyclic. Then $N$ is hereditarily conjugacy separable and has decidable conjugacy problem.
\end{theorem}

Corollary \ref{cor:mainRAAGs} tells us that these quotients are not only virtually polycyclic, but they are actually virtually abelian. Hence, we get:

\begin{corollary}\label{corollary3}
Any finitely generated full normal subgroup of a RAAG is hereditarily conjugacy separable, so in particular, conjugacy separable.
\end{corollary}

We finish the note by pointing out that since finitely generated normal subgroups of RAAGs are finite index subgroups of kernels of characters, we can algorithmically determine their finiteness conditions:

\begin{corollary}\label{corollary4}
Let $G$ be a RAAG given by its standard presentation and let $N< G$ be a finitely generated full normal subgroup of $G$. Then there is an algorithm that given a finite generating set of $N$ as words in the generators of $G$ and $n\in \mathbb N$, decides whether or not $N$ is of type $F_n$ and $FP_n$.
\end{corollary}

\begin{proof}
We have a presentation of $G \slash N$ by adding the generators of $N$ as relations to the standard presentation of $G$. By Corollary \ref{cor:mainRAAGs}, $G \slash N$ is finite-by-abelian, so by using Tietze transformations we may assume that the presentation that we have is actually a finite-by-abelian presentation.

Then, there is a homomorphism $\chi$ from $G$ to the abelian group and $N$ has finite index in the kernel, so it suffices to algorithmically decide the finiteness properties of the kernel of  $\chi$. For that, by \cite{MeierMeinertVanWyk}, in order to determine if $\ker \chi$ is of type $F_n$ (respectively of type $FP_n$), it is enough to construct the flag subcomplex $\mathcal{L}_{\chi}$ defined by the living vertices and decide whether or not $\mathcal{L}_{\chi}$ is $(n-1)$-connected (respectively $(n-1)$-acyclic) and $(n-1)$-$\mathbb Z$-acyclic dominating (see definitions in \cite{MeierMeinertVanWyk}).
\end{proof}

\end{document}